\documentclass[a4paper, 11pt]{article}
\usepackage{amsmath,amstext, amsthm, amssymb}

\usepackage[all]{xy}

\newtheorem{theo}{Theorem}
\newtheorem{lemma}[theo]{Lemma}

\newtheorem{cor}[theo]{Corollary}
\newtheorem{defi}[theo]{Definition}

\newtheorem*{theononumber}{Theorem}

\theoremstyle{definition}
\newtheorem{ex}[theo]{Example}

\def\q{\mathfrak{q}}
\def\s{\sigma}

\def\Gl{\operatorname{Gl}}

\def\ida{\mathfrak{a}}
\newcommand{\idb}{\mathfrak{b}}
\def\m{\mathfrak{m}}

\def\pa{\partial}

\def\<{\langle}
\def\>{\rangle}

\def\de{\delta}
\def\sdp{\Sigma\Delta\partial}
\def\sdP{\Sigma\Delta\Pi}
\def\sd{{\Sigma\Delta}}

\author{Michael Wibmer}

\title{Existence of $\partial$-parameterized Picard-Vessiot extensions over fields with algebraically closed constants}

\begin{document}
\maketitle

\begin{abstract}
The purpose of this short note is to establish the existence of \mbox{$\pa$-parameterized} Picard-Vessiot extensions of systems of linear difference-differential equations
over difference-differential fields with algebraically closed constants.
\end{abstract}

\renewcommand{\labelenumi}{{\rm (\roman{enumi})}}

\section*{Introduction}

A $\Pi$-parameterized Picard-Vessiot theory of linear differential equations has been introduced in \cite{CassidySinger:GaloisTheoryofParameterizedDifferentialEquations}. Here $\Pi=\{\pa_1,\ldots,\pa_m\}$ is a finite set of commuting derivations. In \cite{HardouinSinger:DifferentialGaloisTheoryofLinearDifferenceEquations} this theory has been extended to include linear difference equations. In this article we shall only be concerned with the case of one derivation parameter, i.e. $\Pi=\{\pa\}$ or $m=1$.

The theories in \cite{CassidySinger:GaloisTheoryofParameterizedDifferentialEquations} and \cite{HardouinSinger:DifferentialGaloisTheoryofLinearDifferenceEquations} are build under the assumption that the $\Sigma\Delta$-constants of the base $\sdP$-field are $\Pi$-algebraically closed. Here $\Sigma$ is a set of automorphism, $\Delta$ is a set derivations, all the operators are assumed to commute and one is interested in the solutions of systems of linear difference-differential equations in the operators $\Sigma$ and $\Delta$. The Galois groups, which are $\Pi$-algebraic groups over the $\sd$-constants, measure the $\Pi$-algebraic relations among the solutions.

The assumption that the $\sd$-constants are $\Pi$-algebraically closed is quite dissatisfactory because:
\begin{itemize}
\item There are no natural examples of $\Pi$-algebraically closed fields.\footnote{While this statement expresses a subjective opinion rather than a solid mathematical fact, it at least seems to be a commonly accepted opinion. See e.g. \cite[end of first paragraph of Section 6.2]{Poizat:ACourseinModelTheory}.}
\item In the concrete applications, one always has to ``tensor things up'' to a $\Pi$-algebraic closure to be able to apply the theory and then one needs to find an ad hoc descent argument to get back to the situation originally of interest.
\end{itemize}

The main result of this article is the following theorem. (See Definition \ref{defi paPV} below for the definition of $\pa$-parameterized Picard-Vessiot rings.)

\begin{theononumber}
Let $K$ be a $\Sigma\Delta\pa$-field of characteristic zero. Then for every $\sd$-linear system over $K$ there exists a $\pa$-parameterized Picard-Vessiot ring $R$ such that the $\sd$-constants of $R$ are a finite algebraic field extension of the $\sd$-constants of $K$.
\end{theononumber}

With the above theorem in hand one can essentially remove the assumption of $\pa$-algebraically closed constants from the theories in \cite{CassidySinger:GaloisTheoryofParameterizedDifferentialEquations} and \cite{HardouinSinger:DifferentialGaloisTheoryofLinearDifferenceEquations}. This is carried out in more detail in \cite{Ovchinkiovetal:ParameterizedPicardVessiotExtensionsandAtiyahExtensions} (for the case $\Sigma=\emptyset$) and \cite[Section 1.2]{diVizioHardouin:Descent} (for the case $\Sigma=\{\sigma\}$ and $\Delta=\emptyset$). Of course one will have to accept some modifications:
\begin{itemize}
\item The uniqueness theorem has to be reformulated: If $R_1$ and $R_2$ are $\pa$-parameterized Picard-Vessiot rings over the base $\sdp$-field $K$ for the same system of equations then there exists a finitely $\pa$-generated constrained $\pa$-field extension $k$ of the $\sd$-constants $K^\sd$ which contains $R_1^\sd$ and $R_2^\sd$ such that $R_1\otimes_{R_1^\sd}k$ and $R_2\otimes_{R_2^\sd}k$ are isomorphic. (Recall from \cite{Kolchin:constrainedExtensions} that a finitely $\pa$-generated $\pa$-field extension is constrained if and only if it embeds into the differential closure of the base field.) Cf. \cite[Theorem 2.19, p. 28]{Wibmer:Chevalley}.

\item One can not naively identify the Galois group with the set of its $K^\sd$-rational points. Instead one should rather adopt a scheme theoretic point of view and consider functorially defined invariants. In particular one should define the Galois group by exhibiting an appropriate automorphism functor and showing that it is representable by a $\pa$-Hopf algebra. Cf. \cite{Dyckerhoff:InverseProblem}, \cite{Maurischat:GaloisTheoryforIterativeConnectionsandNonreducedGaloisGroups} and \cite{AmanoMasuokaTakeuchi:HopfPVtheory}.
\end{itemize}

The idea of the construction of the $\pa$-Picard-Vessiot ring in the above theorem has been inspired by the theory of difference kernels developed by R. Cohn \cite[Chapter 6]{Cohn:difference}. See \cite{Lando:JacobiBound} for a differential version. The method of the construction originated in \cite{Wibmer:Chevalley} where a similar result is obtained for the case of $\s$-parameterized linear differential equations, i.e. for the case of a difference parameter $\s$ instead of a differential parameter $\pa$ as in this text. See \cite[Lemma 2.16, p. 27]{Wibmer:Chevalley}. The method is also used in \cite{diVizioHardouin:Descent} (for the case $\Sigma=\{\sigma\}, \Delta=\emptyset$ and $\Pi=\{\pa\}$).

An approach to parameterized (and classical) Picard-Vessiot theory over fields with not necessarily $\Pi$-algebraically closed fields of constants (for the case $\Sigma=\emptyset$) based on the tannakian machinery can be found in \cite{Ovchinkiovetal:ParameterizedPicardVessiotExtensionsandAtiyahExtensions}. The existence results in \cite{Ovchinkiovetal:ParameterizedPicardVessiotExtensionsandAtiyahExtensions} are in a certain sense more general, in particular they are not restricted to the case $\Pi=\{\pa\}$. However, due to its  simple and elementary nature, the construction of the $\pa$-parameterized Picard-Vessiot extension presented in this article, seems to be of an independent interest.

Another approach to parameterized Picard-Vessiot theory over fields with not necessarily $\Pi$-algebraically closed fields of constants has been presented by A. Nieto in \cite{Nieto:OnSigmaDeltaPicardVessiotExtension} (for the case $\Sigma=\{\sigma\}, \Delta=\emptyset$ and $\Pi=\{\pa\}$). The Picard-Vessiot extensions considered in \cite{Nieto:OnSigmaDeltaPicardVessiotExtension} usually do introduce a somewhat large field of new constants.

\bigskip

\section*{Existence of $\pa$-parameterized Picard-Vessiot \mbox{extensions}}

All rings are assumed to be commutative and all fields are assumed to be of characteristic zero.
We start by recalling the basic setup from \cite{CassidySinger:GaloisTheoryofParameterizedDifferentialEquations} and \cite{HardouinSinger:DifferentialGaloisTheoryofLinearDifferenceEquations}.

Let $\Sigma=\{\s_i\}$ and $\Delta=\{\de_i\}$ denote arbitrary sets of symbols. The case that $\Sigma$ or $\Delta$ is empty is allowed but we exclude the case that both $\Sigma$ and $\Delta$ are empty. By a \emph{$\sdp$-ring} one means a ring $R$ together with ring automorphisms $\s:R\rightarrow R$ for every $\s\in\Sigma$, derivations $\de:R\rightarrow R$ for every $\de\in\Delta$ and a derivation $\pa:R\rightarrow R$ such that all the operators commute, i.e. $\mu(\tau(r))=\tau((\mu(r))$ for all $r\in R$ and all $\mu,\tau\in \Sigma\cup\Delta\cup\{\pa\}$. As illustrated in \cite{AmanoMasuokaTakeuchi:HopfPVtheory} the action of the operators can quite conveniently be summarized by a coalgebra action but we shall not adopt this point of view here. A $\sdp$-field is a $\sdp$-ring whose underlying ring is a field. A morphism of $\sdp$-rings is a morphism of rings that commutes with all the operators. There are some further obvious notions like $\sdp$-ideal, $K$-$\sdp$ algebra,... that generalize the well-known concepts from difference or differential algebra. The constants are denoted with superscripts, for example if $R$ is a $\sdp$-ring then
\[R^\sd=\{r\in R |\ \s(r)=r \ \forall\ \s\in\Sigma \text{ and } \delta(r)=0 \ \forall\ \delta\in \Delta\}\] denotes the ring of $\sd$-constants.

 A \emph{$\sd$-linear system} over a $\sdp$-field $K$ is a system of equations of the form
\begin{equation}\label{linear system}
\begin{aligned}
\s_i(Y)&=A_iY\ \ A_i\in\Gl_n(K),\ \s_i\in\Sigma \\
\de_i(Y)&=B_iY\ \ B_i\in K^{n\times n},\ \de_i\in\Delta
\end{aligned}
\end{equation}
where the $A_i$ and $B_j$ satisfy the integrability conditions
\begin{equation} \label{integrability conditions}
\begin{aligned}
 \s_i(A_j)&=\s_j(A_i)A_j\\
\s_i(B_j)A_i&=\de_j(A_i)+A_iB_j\\
\de_i(A_j)+A_jA_i&=\de_j(A_i)+A_iA_j
\end{aligned}
\end{equation}
for all $\s_i,\s_j\in\Sigma$ and all $\de_i,\de_j\in\Delta$.
The above integrability conditions express the property that $\mu(\tau(Z))=\tau(\mu(Z))$ for all $\mu,\tau\in\Sigma\cup\Delta$ for a solution matrix $Z\in\Gl_n(R)$ in some $K$-$\Sigma\Delta$-algebra $R$.

There are two special cases of crucial importance: The first one is $\Sigma=\emptyset$ and $\Delta=\{\de\}$ in which case one is simply looking at a linear differential equation
\[\de(Y)=BY,\ \ B\in K^{n\times n}.\]
The second one is $\Sigma=\{\s\}$ and $\Delta=\emptyset$ in which case one is simply looking at a linear difference equation
\[\s(Y)=AY, \ \ A\in\Gl_n(K).\]

\bigskip

In the (usual) Picard-Vessiot theory of linear difference equations one needs to admit solution rings which are slightly more general than fields to be able to associate an appropriate ``splitting field'' (the Picard-Vessiot extension) to \emph{every} linear difference equation (\cite{SingerPut:difference}). These ``minimal solution fields'' are not quite fields but rather finite direct products of fields. Because of their relevance for the subject we think they deserve a name of their own:

\begin{defi}
A \emph{$\Sigma$-pseudo field} is a Noetherian $\Sigma$-simple $\Sigma$-ring such that every non-zero divisor is invertible.
\end{defi}
It is not too hard to see that a $\Sigma$-pseudo field is a finite direct product of fields with a transitive action of the group generated by $\Sigma$ on the factors. See e.g. \cite[Cor. 1.16, p. 12]{SingerPut:difference} or \cite[Prop. 11.5, p. 162]{AmanoMasuokaTakeuchi:HopfPVtheory}. Pseudo-fields are also used in \cite{AmanoMasuoka:artiniansimple}, \cite{Wibmer:thesis}, \cite{Wibmer:Chevalley}, \cite{Trushin:DifferenceNullstellsatzCaseOfFiniteGroup} and \cite{Trushin:DifferenceNullstellensatz}. The pseudo-fields in \cite{Trushin:DifferenceNullstellensatz} are somewhat more general; they need not be Noetherian.

\begin{defi}
A $\sdp$-ring $L$ is called a \emph{$\Sigma$-pseudo $\Delta\pa$-field} if $(L,\Sigma)$ is a $\Sigma$-pseudo field.
\end{defi}

If $\Sigma=\emptyset$ then a $\Sigma$-pseudo $\Delta\pa$-field is simply a $\Delta\pa$-field. (A simple ring is a field.)
We recall the fundamental notions of $\pa$-Picard-Vessiot extension and $\pa$-Picard-Vessiot ring from
\cite[Def. 9.4, p. 140]{CassidySinger:GaloisTheoryofParameterizedDifferentialEquations} and \cite[Def. 6.10, p. 363]{HardouinSinger:DifferentialGaloisTheoryofLinearDifferenceEquations}\footnote{This notation is not consistent with \cite{HardouinSinger:DifferentialGaloisTheoryofLinearDifferenceEquations}. The notation in use here however has been proposed by one of the authors of \cite{HardouinSinger:DifferentialGaloisTheoryofLinearDifferenceEquations} and appears to be somewhat more practical. See \cite{diVizioHardouin:Descent}.}.

\begin{defi} \label{defi paPV}
Let $K$ be a $\sdp$-field. A \emph{$\pa$-Picard-Vessiot extension} for the $\sd$-linear system (\ref{linear system}) is a $\Sigma$-pseudo $\Delta\pa$-field extension $L$ of $K$ such that
\begin{enumerate}
\item $L$ is $\pa$-generated by a fundamental solution matrix for (\ref{linear system}), i.e. there exists $Z\in\Gl_n(L)$ such that $\s_i(Z)=A_iZ$, $\de_i(Z)=B_iZ$ for all $\s_i\in\Sigma$, $\de_i\in \Delta$ and $L=K\langle Z_{ij};\ 1\leq i,j\leq n\rangle_\pa$. \footnote{Here we use $K\langle Z_{ij};\ 1\leq i,j\leq n\rangle_\pa$ to denote the smallest $K$-$\pa$-subalgebra of $L$ which is closed under taking inverses and contains the $Z_{ij}$'s. If $L$ is a field (for example if $\Sigma=\emptyset$) this simply means that $L$ is generated by the entries of $Z$ as a $\pa$-field extension of $K$.}
\item $L^\sd=K^\sd$.
\end{enumerate}

By a \emph{$\pa$-Picard-Vessiot ring} or \emph{$\pa$-parameterized Picard-Vessiot ring} for the $\Sigma\Delta$-linear system (\ref{linear system}) we mean a $K$-$\sdp$-algebra $R$ such that
\begin{enumerate}
\item $R$ is $\pa$-generated by a fundamental solution matrix for (\ref{linear system}), i.e. there exists $Z\in\Gl_n(R)$ such that $\s_i(Z)=A_iZ$, $\de_i(Z)=B_iZ$ for all $\s_i\in\Sigma$, $\de_i\in \Delta$ and as a $K$-$\pa$-algebra $R$ is generated by the entries of $Z$ and the inverse of the determinant of $Z$. (In formulas this is expressed as $R=K\{Z,\frac{1}{\det(Z)}\}_\pa$.)
\item $R$ is $\Sigma\Delta$-simple.
\end{enumerate}
\end{defi}

The above definition of $\pa$-Picard-Vessiot rings slightly deviates from the definition in \cite{HardouinSinger:DifferentialGaloisTheoryofLinearDifferenceEquations} and  \cite{CassidySinger:GaloisTheoryofParameterizedDifferentialEquations} where (ii) is replaced by

\begin{enumerate}
\setcounter{enumi}{1}
\renewcommand{\labelenumi}{{\rm (\roman{enumi})'}}
 \item R is $\sdp$-simple.
\renewcommand{\labelenumi}{{\rm (\roman{enumi})}}
\end{enumerate}

However in \cite[Cor. 6.22, p. 372]{HardouinSinger:DifferentialGaloisTheoryofLinearDifferenceEquations} it is shown that if $K^\sd$ is $\pa$-algebraically closed then (i) and (ii)' imply (ii). Since clearly a $\sd$-simple ring is $\sdp$-simple this shows that (ii) and (ii)' are equivalent.
Thus, under the assumption of $\pa$-algebraically closed $\sd$-constants -- which is the standard assumption in \cite{CassidySinger:GaloisTheoryofParameterizedDifferentialEquations} and \cite{HardouinSinger:DifferentialGaloisTheoryofLinearDifferenceEquations} -- there is no real difference between (ii) and (ii)'.

In all generality, i.e. if the $\sd$-constants need not be $\pa$-algebraically closed, the choice of (ii) over (ii)' is motivated by the following facts:

\begin{itemize}
\item One can always satisfy (ii). (See Theorem \ref{theo main} below.)
\item If $L$ is a $\pa$-Picard-Vessiot extension with fundamental solution matrix $Z$ then $R=K\{Z,\frac{1}{\det(Z)}\}_\pa$ is a $\pa$-Picard-Vessiot ring, i.e. satisfies (ii).
\item If a $K$-$\sdp$ algebra $R$ satisfies (i) and has the natural property that $K\left[Z,\pa(Z),\ldots,\pa^d(Z),\tfrac{1}{\det(Z)}\right]$ is a (usual) Picard-Vessiot ring for every $d\geq 0$ (cf. \cite[Prop. 6.21]{HardouinSinger:DifferentialGaloisTheoryofLinearDifferenceEquations}) then $R$ automatically satisfies (ii).
\item Condition (ii) in the definition of $\pa$-Picard-Vessiot extensions is also just a condition on $\sd$ and not on $\sdp$.
\item In the analogous setting with $\pa$ replaced by $\sigma$, i.e. in the study of linear differential equations with a difference parameter \emph{only} the analog of (ii) leads to a satisfactory theory.
\item If the $\sd$-constants are not $\pa$-algebraically closed, condition (ii) simply seems more practical. For example from (ii)' it is not even clear why $R^\sd$ should be a field.
\end{itemize}

\bigskip

\bigskip

For the proof of the main theorem we shall need three preparatory lemmas and some more notation:

Let $K$ be a $\pa$-field. The $\pa$-polynomial ring $K\{x\}=K\{x\}_\pa$ in the $\pa$-variables $x=(x_1,\ldots,x_n)$ over $K$ is the polynomial ring in the infinitely many
variables $(\pa^d(x_i))_{i=1,\ldots,n, d\ge 0}$ over $K$ with the natural action of $\pa$. For $d\geq 0$ let $K\{x\}_d$ denote the set of all differential polynomials of order lesser or equal to $d$, i.e.
\[K\{x\}_d=K[x,\pa(x),\ldots,\pa^d(x)]\subset K\{x\}.\]
For consistency reasons we set $K\{x\}_{-1}=K$.
Then $K\{x\}_d$ is a $K$-sub algebra of $K\{x\}$, $\cup_{d\geq0}K\{x\}_d=K\{x\}$ and $\pa\colon K\{x\}\to K\{x\}$ restricts to a derivation $\pa\colon K\{x\}_{d-1}\to K\{x\}_d$.

The following lemma is due to B. Lando (\cite[Prop. 1, p. 121]{Lando:JacobiBound}). For the convenience of the reader and the sake of completeness we include the proof.
\begin{lemma} \label{lemma lando}
Let $d\geq 0$ and let $\q\subset K\{x\}_d$ be a prime ideal such that $\pa(\q\cap K\{x\}_{d-1})\subset\q$. Then there exists a prime ideal $\q'$ of $K\{x\}_{d+1}$ such that
$\q'\cap K\{x\}_d=\q$ and $\pa(\q)\subset\q'$.
\end{lemma}
\begin{proof}
Set $\q''=\q\cap K\{x\}_{d-1}$. Let $k(\q)=K\left(a,\pa(a),\ldots,\pa^{d}(a)\right)$ denote the residue field of $\q\subset K\{x\}_d$. Here we use $\pa^j(a)$ to denote the image of $\pa^j(x)=(\pa^j(x_1),\ldots,\pa^j(x_n))$ in $k(\q)$ for $j=0,\ldots,d$.
Similarly let $k(\q'')=K\left(a,\ldots,\pa^{d-1}(a)\right)$ denote the residue field of $\q''\subset K\{x\}_{d-1}$. Then
$k(\q'')\subset k(\q)$ and by assumption we have a well-defined derivation
$\pa\colon k(\q'')\rightarrow k(\q)$ satisfying $\pa(\pa^j(a))=\pa^{j+1}(a)$ for $j=0,\ldots,d-1$.
We can extend it to a derivation $\pa\colon k(\q)\to k(\q)$ (\cite[Cor. 1, Chapter V, \S 16, No. 3, A.V.130]{Bourbaki:Algebra2}).

Let $\psi\colon K\{x\}_{d+1}\to k(\q)$ denote the $K$-algebra morphism determined by $\psi(\pa^{j}(x))=\pa^j(a)$ for $j=0,\ldots,d+1$. By construction the diagram
\[
\xymatrix{
K\{x\}_{d}\ar[r]^-{\pa}  \ar[d]_\psi & K\{x\}_{d+1} \ar[d]^\psi \\
k(\q) \ar[r]^-{\pa} & k(\q)
}
\]
is commutative. It follows that $\q'=\ker{\psi}\subset K\{x\}_{d+1}$ satisfies $\q'\cap K\{x\}_d=\q$ and $\pa(\q)\subset\q'$.
\end{proof}

\begin{lemma} \label{lemma main}
Let $d\geq 0$ and let $\ida\subset K\{x\}_d$ be a radical ideal such that $1\notin\ida$ and $\pa(\ida'')\subset\ida$, where $\ida''=\ida\cap K\{x\}_{d-1}$. Let $\idb$ denote the ideal of $K\{x\}_{d+1}$ generated by $\ida$ and $\pa(\ida)$.

Then for every prime ideal $\q''$ of $K\{x\}_{d-1}$ that is minimal above $\ida''$ there exists a prime ideal $\q'$ of $K\{x\}_{d+1}$ such that $\q''=\q'\cap K\{x\}_{d-1}$, $\q=\q'\cap K\{x\}_d$ is a minimal prime ideal above $\ida$ and $\pa(\q)\subset\q'$. Moreover $\idb\subset\q'$, in particular $1\notin\idb$.
\end{lemma}
\begin{proof}
Set $\ida''=\ida\cap K\{x\}_{d-1}$ and let $\q''\subset K\{x\}_{d-1}$ be a minimal prime ideal above $\ida''$. There exists a minimal prime ideal $\q\subset K\{x\}_{d}$ above $\ida$ with $\q\cap K\{x\}_{d-1}=\q''$.

We will show that $\pa(\q'')\subset\q$. So let $p\in\q''$. Because $\q''$ is minimal above $\ida''$ there exists a $q\in K\{x\}_{d-1}$, $q\notin\q''$ such that $pq\in\ida''$. By assumption we have $\pa(pq)\in\ida$. Because $\pa(pq)=p\pa(q)+\pa(p)q$
it follows from $p,\pa(pq)\in\q$ that $\pa(p)q\in\q$. As $q\notin\q$ this implies $\pa(p)\in\q$.

Thus $\pa(\q'')\subset\q$ and we can apply Lemma \ref{lemma lando} to obtain a prime ideal $\q'$ of $K\{x\}_{d+1}$ with $\q'\cap K\{x\}_d=\q$ and $\pa(\q)\subset\q'$. Then $\ida\subset\q\subset\q'$ and $\pa(\ida)\subset\pa(\q)\subset\q'$. Consequently $\idb\subset \q'$.
\end{proof}

It is interesting to note that Lemma \ref{lemma main} does not generalize to the partial case (i.e. to the case of several \emph{commuting} derivations) in an obvious fashion. This is due to the existence of the so-called integrability conditions. This is the main reason why this article is restricted to the case of one derivation parameter, i.e. $\Pi=\{\pa\}$.

We illustrate the phenomena with a simple example. (Cf. \cite[Ex. 2.3.9, p. 34]{Seiler:Involution}.)

\begin{ex}
Let $K=\mathbb{C}(u,v)$ denote the field of rational functions in two variables $u$ and $v$ over the field of complex numbers. Let $\Pi=\{\pa_1,\pa_2\}$ and consider $K$ as $\Pi$-field with the standard derivations, i.e. $\pa_1=\frac{\operatorname{d}}{\operatorname{d}u}$ and $\pa_2=\frac{\operatorname{d}}{\operatorname{d}v}$. Let $x=x_1$ denote a single $\Pi$-variable over $K$. Clearly the ideal $\ida$ of $K\{x\}_1=K[x,\pa_1(x),\pa_2(x)]$ generated by $v\pa_1(x)+1$ and $\pa_2(x)$ is radical and does not contain $1$. We have
$\ida''=\ida\cap K\{x\}_0=\ida\cap K[x]=\{0\}$. Thus $\pa_1(\ida'')\subset\ida$ and $\pa_2(\ida'')\subset\ida$.
Let $\idb$ denote the ideal of
\[K\{x\}_2=K[x,\pa_1(x),\pa_2(x),\pa_1^2(x),\pa_1\pa_2(x),\pa_2^2(x)]\]
generated by $\ida$, $\pa_1(\ida)$ and $\pa_2(\ida)$. Then $\pa_2(v\pa_1(x)+1)=v\pa_1\pa_2(x)+\pa_1(x)\in\idb$ and $\pa_1(\pa_2(x))=\pa_1\pa_2(x)\in\idb$. Therefore also $\pa_1(x)\in\idb$. But since $v\pa_1(x)+1\in\ida\subset\idb$ this gives $1\in\idb$.
\end{ex}

We shall also need the following lemma.

\begin{lemma} \label{lemma finite}
Let $K$ be a $\sdp$-field and $R$ a finitely $\pa$-generated $\sd$-simple $K$-$\sdp$-algebra. Assume that $R^\sd$ is algebraic over $K^\sd$. Then $R^\sd$ is a finite field extension of $K^\sd$.
\end{lemma}
\begin{proof}
We know from \cite[Lemma 6.8, p. 361]{HardouinSinger:DifferentialGaloisTheoryofLinearDifferenceEquations} or \cite[Prop. 2.4, p. 750]{AmanoMasuoka:artiniansimple} that $R$ has a very special structure: There exist orthogonal idempotent elements $e_1,\ldots,e_t\in R$ such that
\begin{itemize}
\item $R=e_1R\oplus\cdots\oplus e_tR$.
\item The $e_iR$'s are integral domains stable under $\Delta$, $\pa$ and $\widetilde{\Sigma}=\{\s^t|\ \s\in\Sigma\}$. Moreover $e_iR$ is $\widetilde{\Sigma}\Delta\partial$-simple.
\item Let $G$ be the group generated by all the automorphisms $\s\in\Sigma$. Then the action of $G$ on $\{e_1R,\ldots,e_tR\}$ is transitive. In particular $G_1=\{g\in G|\ g(e_1R)=e_1R\}$ is a subgroup of $G$ of index $t$.
\end{itemize}

The map $R^\sd\to (e_1R)^{\widetilde{\Sigma}\Delta}$ given by $r\mapsto e_1r$ is an isomorphism of fields (cf. \cite[Lemma 1.6, p. 748]{AmanoMasuoka:artiniansimple} ). (The inverse is given by $s\mapsto \sum_{i=1}^t g_i(s)$ where $g_1,\ldots,g_t$ is some system of representatives of $G/G_1$.)

Summarily this shows that we can assume without loss of generality that $R$ is an integral domain. The advantage of this reduction is that we can now consider the quotient field $L$ of $R$. This is naturally a $\sdp$-field extension of $K$, finitely $\pa$-generated as a $\pa$-field extension of $K$. Let $M=KL^\sd\subset L$ denote the field compositum of $K$ and $L^\sd$ inside $L$. Then $M$ is an intermediate $\sdp$-field of $L|K$. Because $R$ is $\sd$-simple $R^\sd=L^\sd$. By assumption $R^\sd$ is algebraic over $K^\sd$, so in particular $M$ is algebraic over $K$. But an algebraic intermediate $\pa$-field of a finitely $\pa$-generated $\pa$-field extension is finite (\cite[Cor. 2, Chapter II, Section 11, p. 113]{Kolchin:differentialalgebraandalgebraicgroups}). So $M$ is a finite field extension of $K$. Because $K$ and $L^\sd$ are linearly disjoint over $K^\sd$, i.e.
$M=K\otimes_{K^\sd} L^\sd$ (\cite[Lemma 6.11, p. 364]{HardouinSinger:DifferentialGaloisTheoryofLinearDifferenceEquations}) this implies that the field extension $L^\sd|K^\sd$ is finite.
\end{proof}

Now we are prepared to prove the main theorem of this short note.

\begin{theo} \label{theo main}
Let $K$ be a $\Sigma\Delta\pa$-field. Then for every $\sd$-linear system over $K$ there exists a $\pa$-Picard-Vessiot ring $R$ such that $R^{\sd}$ is a finite algebraic field extension of $K^{\sd}$.
\end{theo}
\begin{proof}
We consider a $\sd$-linear system given as in (\ref{linear system}). Let $X=(X_{ij})_{1\leq i,j\leq n}$ be a matrix of $\pa$-indeterminates over $K$ and let $S=K\{X,\frac{1}{\det(X)}\}_\pa$ denote the $\pa$-polynomial ring in the $\pa$-indeterminates $X_{ij}$ $(1\leq i,j\leq n)$ localized at the multiplicatively closed subset generated by $\det(X)$. We can turn $S$ into a $K$-$\sdp$-algebra by postulating that
\begin{align*}
\s_i(X)&=A_iX\ \text{ for all } \s_i\in\Sigma, \\
\de_i(X)&=B_iX\ \text{ for all } \de_i\in\Delta
\end{align*}
and that all elements of $\Sigma\cup\Delta$ commute with $\pa$ (cf. \cite[Remark after Def. 6.10, p. 363]{HardouinSinger:DifferentialGaloisTheoryofLinearDifferenceEquations}). Note that the integrability conditions (\ref{integrability conditions}) assure that also any two elements of $\Sigma\cup\Delta$ commute.

For $d\geq 0$ let
\[S_d=K\left[X,\pa(X),\ldots,\pa^d(X),\tfrac{1}{\det(X)}\right]\subset S \]
denote the subring of all $\pa$-polynomials of order less than or equal to $d$ localized at $\det(X)$. Note that $S_d$ is a $K$-$\sd$-subalgebra of $S$, $S=\cup_{d\geq 0} S_d$ and
$\pa\colon S\to S$ restricts to a derivation $\pa\colon S_{d-1}\rightarrow S_d$.

Similarly, let $K\{X\}_d=K[X,\ldots,\pa^d(X)]\subset S_d$ denote the $K$-$\sd$-algebra of $\pa$-polynomials of order less than or equal to $d$.

\bigskip

We will show by induction on $d\geq 0$ that there exists a sequence $(\m_d)_{d\geq 0}$ with the following properties:
\begin{enumerate}
\item $\m_d$ is a $\sd$-maximal ideal of $S_d$, i.e. a maximal element in the set of all $\sd$-ideals of $S_d$ ordered by inclusion.
\item $\m_d\cap S_{d-1}=\m_{d-1}$.
\item $\pa(\m_{d-1})\subset \m_d$.
\end{enumerate}
For $d=0$ conditions (ii) and (iii) are understood to be empty. So we can choose any $\sd$-maximal ideal $\m_0$ of $S_0=K[X,\tfrac{1}{\det(X)}]$ (which exists by Zorn's Lemma).

Assume now that the sequence $\m_0,\ldots,\m_{d}$ has already been constructed. Let $\idb\subset S_{d+1}$ denote the ideal of $S_{d+1}$ generated by $\m_{d}$ and $\pa(\m_{d})$. The crucial point now is to show that $1\notin\idb$. For this we can use Lemma \ref{lemma main} as follows: Set $\ida=\m_d\cap K\{X\}_d$ and
$\ida''=\ida\cap K\{X\}_{d-1}=\m_{d-1}\cap K\{X\}_{d-1}$. Because a $\sd$-maximal ideal is radical (\cite[Lemma 6.7, p. 361]{HardouinSinger:DifferentialGaloisTheoryofLinearDifferenceEquations}) $\m_d$ and $\ida$ are radical ideals. By construction $\pa(\ida'')\subset\ida$. Thus we can apply Lemma \ref{lemma main} to conclude that the ideal $\widetilde{\idb}$ of $K\{X\}_{d+1}$ generated by $\ida$ and $\pa(\ida)$ does not contain $1$.

Suppose that $1\in\idb$. Since $\idb$ agrees with the ideal of $S_{d+1}$ generated by $\widetilde{\idb}$ this implies that there exists an integer $m\geq 1$ such that $\det(X)^m\in\widetilde{\idb}$. By Lemma \ref{lemma main} there exists a prime ideal $\q'$ of $K\{X\}_{d+1}$ with $\widetilde{\idb}\subset\q'$ such that $\q=\q'\cap K\{X\}_d$ is a minimal prime above $\ida$. But then it follows from $\det(X)^m\in\widetilde{\idb}$ that $\det(X)^m\in\q$. As $\q$ is minimal above $\ida$ this implies that there exist $p\in K\{X\}_d$, $p\notin\ida$ such that $p\det(X)^m$ lies in $\ida$. This is in contradiction to $\ida=\m_{d}\cap K\{X\}_{d}$, i.e. $\ida$ is satured with respect to $\det(X)$.

Therefore $1\notin\idb$. By construction $\idb$ is a $\sd$-ideal of $S_{d+1}$. Let $\m_{d+1}\subset S_{d+1}$ denote a $\sd$-maximal $\sd$-ideal of $S_{d+1}$ containing $\idb$. Then $\m_{d+1}\cap S_d$ is a $\sd$-ideal of $S_d$ containing $\m_d$. As $\m_d$ is $\sd$-maximal this gives $\m_{d+1}\cap S_d=\m_d$. By construction we have $\pa(\m_d)\subset \idb\subset\m_{d+1}$. Thus $\m_{d+1}$ has the desired properties and the existence of the sequence $(\m_d)_{d\geq 0}$ is established.

\bigskip

Now we can set $\m=\cup_{d\geq 0}\m_d\subset S$. Because of condition (iii) $\m$ is a $\sdp$-ideal of $S$. Moreover it follows easily from condition (i) that $\m$ is $\sd$-maximal.
Therefore the quotient ring $R=S/\m$ is a $\sd$-simple $K$-$\sdp$-algebra. By construction $R$ satisfies condition (i) of Definition \ref{defi paPV}, i.e. $R$ is a $\pa$-Picard-Vessiot ring for our $\sd$-linear system (\ref{linear system}).

Next we will show that $R^\sd$ is algebraic over $K^\sd$. In general if $K$ is a $\sd$-field and $T$ a $\sd$-simple $K$-$\sd$-algebra that is finitely generated as a $K$-algebra then $T^\sd$ is an algebraic field extension of $K^\sd$. In all generality there appears to be no suitable reference for this result in the literature. However, the standard proof (based on Chevalley's theorem) can easily be adopted. (See \cite[Prop. 13.7, p. 4491]{Kovacic:differentialgaloistheoryofstronglynormal} for the case $\Sigma=\emptyset$, $\Delta=\{\delta_1,\ldots,\delta_m\}$ and $K^\Delta$ algebraically closed; \cite[Lemma 1.8, p. 7]{SingerPut:difference} for the case $\Sigma=\{\sigma\}$, $\Delta=\emptyset$ and $K^\sigma$ algebraically closed; \cite[Theorem 4.4, p. 505]{Takeuchi:hopfalgebraicapproach} for the $C$-ferential case.)
It follows from this result that $(S_d/\m_d)^\sd$ is an algebraic field extension of $K^\sd$ for every $d\geq 0$. Since $R$ can be interpreted as the union of all the $S_d/\m_d$'s this shows that $R^\sd$ is an algebraic field extension of $K^\sd$.

It remains to see that $R^\sd$ is a finite extension of $K^\sd$. But this is clear from Lemma \ref{lemma finite}.
\end{proof}

\begin{cor} \label{cor main1}
Let $K$ be $\sdp$-field and consider a $\sd$-linear system of the form (\ref{linear system}) over $K$. Then there exists a finite algebraic $\pa$-field extension $\widetilde{k}$ of $k=K^\sd$ such that
there exists a $\pa$-Picard-Vessiot extension for (\ref{linear system}) over $\widetilde{K}=K\otimes_k\widetilde{k}$.
\end{cor}

Here one considers $\widetilde{k}$ as a constant $\sd$-field. First of all we note that in general $\widetilde{K}$ need not be a field. However, $\widetilde{K}$ is a $\Sigma$-pseudo $\Delta\pa$-field (cf. \cite[Prop. 1.4.15, p. 15 and Lemma 1.6.8, p. 24]{Wibmer:thesis}) and there is no obstacle in generalizing the definition of $\pa$-Picard-Vessiot extensions (Definition \ref{defi paPV}) from base $\sdp$-fields to base $\Sigma$-pseudo $\Delta\pa$-fields. In fact, if one wants to make sense of the statement that $L$ is a $\pa$-Picard-Vessiot extension of $L^H$ where $H$ is a closed $\Pi$-subgroup of the Galois group of the $\pa$-Picard-Vessiot extension $L$ of $K$, then one will need to make the definition in this generality.
Also note that $\widetilde{K}$ is a field if $\Sigma=\emptyset$ because then $k$ is relatively algebraically closed in $K$ (\cite[Corollary, Chapter II, Section 4, p. 94]{Kolchin:differentialalgebraandalgebraicgroups}).

\begin{proof}[Proof of Corollary \ref{cor main1}] By Theorem \ref{theo main} there exists a $\sdp$-Picard-Vessiot ring $R$ for (\ref{linear system}) over $K$ such that $\widetilde{k}=R^\sd$ is a finite algebraic field extension of $k=K^\sd$. Let $L$ denote the total ring of quotients of $R$. Then $L$ is naturally a $K$-$\sdp$-algebra. As $R$ is $\sd$-simple also $L$ is $\sd$-simple. Since $R$ is finitely $\pa$-generated over $K$ and $\sd$-simple we know that $R$ is a finite direct product of integral domains with a transitive action of the semigroup generated by $\Sigma$ on the factors (\cite[Lemma 6.8, p. 362]{HardouinSinger:DifferentialGaloisTheoryofLinearDifferenceEquations}). Consequently $L$ is a finite direct product of fields and the action of the semigroup generated by $\Sigma$ on the factors remains transitive. This shows that $L$ is a $\Sigma$-pseudo $\Delta\pa$-field. Since $R$ is $\sd$-simple the usual trick (cf. \cite[Lemma 1.17, p. 13]{SingerPut:differential}) shows that $L^\sd=R^\sd$. Thus $L^\sd=R^\sd=\widetilde{k}=\widetilde{K}^\sd$. This shows that $L$ is a $\pa$-Picard-Vessiot extension of $\widetilde{K}$ for the $\sd$-linear system (\ref{linear system}).
\end{proof}

\begin{cor}
Let $K$ be a $\sdp$-field such that $K^\sd$ is an algebraically closed field. Then for every $\sd$-linear system over $K$ there exists a $\pa$-Picard-Vessiot extension.
\end{cor}
\begin{proof}
Clear from Corollary \ref{cor main1}.
\end{proof}

\bigskip

{\bf Acknowledgments.} I would like to thank Mathias Lederer and all the people from the Kolchin Seminar for their warm hospitality during my trip to the USA in March 2011. Special thanks to Alexey Ovchinnikov, Lucia Di Vizio and Charlotte Hardouin for their assistance in the somewhat labyrinthine birth of this article.

\bibliographystyle{alpha}
\bibliography{bibdata.bib}

\end{document}